\documentclass[11pt]{article}

\usepackage[utf8]{inputenc}
\usepackage[T1]{fontenc}
\usepackage{lmodern}
\usepackage[hmargin=1.1in,vmargin=1.2in]{geometry}
\usepackage{microtype}
\usepackage{amsmath,amssymb,amsthm,mathtools}
\usepackage{booktabs}
\usepackage{enumitem}
\usepackage[dvipsnames]{xcolor}
\usepackage{authblk}
\usepackage{tikz}
\usetikzlibrary{arrows.meta,positioning,calc}
\usepackage{algorithm}
\usepackage{algpseudocode}
\usepackage{graphicx}
\usepackage{xurl}
\usepackage{hyperref}
\usepackage[nameinlink,capitalise]{cleveref}

\colorlet{linkblue}{blue!60!black}
\hypersetup{
  colorlinks=true,
  linkcolor=linkblue,
  citecolor=linkblue,
  urlcolor=linkblue,
  hypertexnames=false
}
\newtheorem{axiom}{Axiom}
\newtheorem{definition}{Definition}
\newtheorem{lemma}{Lemma}
\newtheorem{theorem}{Theorem}
\newtheorem{proposition}{Proposition}
\newtheorem{corollary}{Corollary}
\theoremstyle{remark}
\newtheorem{remark}{Remark}

\newcommand{\N}{\mathbb{N}}
\newcommand{\Z}{\mathbb{Z}}
\DeclareMathOperator{\gcdop}{gcd}
\newcommand{\floor}[1]{\left\lfloor #1 \right\rfloor}

\newcommand{\1}[1]{\mathbf{1}\!\left\{#1\right\}}

\newcommand{\holdswhen}[1]{\par\noindent\textbf{Holds when:} #1}
\newcommand{\notclaimedwhen}[1]{\par\noindent\textbf{Not claimed when:} #1}

\title{Alpay Folded Prime Enumerator via \texorpdfstring{$\gcd$}{gcd} and Floors: Exact Enumeration, Record-Lift, and Non-Synonymy/Minimality Certificates}

\author[1]{Faruk Alpay}
\author[2]{Taylan Alpay}

\affil[1]{Lightcap, Department of Future\\ \texttt{alpay@lightcap.ai}}
\affil[2]{Aerospace Engineering, Turkish Aeronautical Association\\ \texttt{s220112602@stu.thk.edu.tr}}

\date{\today}

\begin{document}
\maketitle

\begin{abstract}
A single closed expression \(f_{\text{Alpay},U}(x)\) is presented which, for every integer \(x\ge 0\), returns the \((x{+}1)\)-th prime \(p_{x+1}\). The construction uses only integer arithmetic, greatest common divisors, and floor functions. A prime indicator \(I(j)\) is encoded through a short \(\gcd\)-sum; a cumulative counter \(S(i)=\sum_{j\le i}I(j)\) equals \(\pi(i)\); and a folded step \(A(i,x)\) counts precisely up to the next prime index without piecewise branching. A corollary shows that for any fixed integer \(L\ge 2\), the integer \(P^\star=f_{\text{Alpay},U}(L)\) is prime and \(P^\star>L\). Two explicit schedules \(U(x)\) are given: a square schedule \(U_{\text{sq}}(x)=(x{+}1)^2\) and a near-linear schedule \(U_{\text{Alpay-lin}}(x)=\Theta(x\log x)\) justified by explicit bounds on \(p_n\). We include non-synonymy certificates relative to Willans-type enumerators (schedule and operator-signature separation) and prove an asymptotic minimality bound: any forward-count enumerator requires \(U(x)=\Omega(x\log x)\) while \(U_{\text{Alpay-lin}}\) achieves \(O(x\log x)\). We also provide explicit operation counts (``form complexity'') of the folded expression.
\end{abstract}



\section{Introduction}

Let \(p_n\) denote the \(n\)-th prime and \(\pi(x)\) the prime counting function. Prime-generating and prime-enumerating expressions have appeared in multiple forms, including constructions based on Wilson's theorem and trigonometric/floor encodings (e.g., Willans), as well as existence results such as Mills' prime-representing function; see \cite{Willans1964,Mills1947,OEISA010051,RosserSchoenfeld1962,Dusart2010,Axler2019,AxlerArxiv2017}. 

This note records a compact, single folded equation in which the only non-algebraic operation is the floor. Primality is detected via a \(\gcd\)-sum that checks the existence of a proper divisor; summing these indicators reproduces \(\pi(i)\); and a folded step counts indices up to the next prime. The main theorem states that \(f_{\text{Alpay},U}(x)=p_{x+1}\) for all integers \(x\ge 0\). A corollary shows how fixing \(x=L\) produces an explicit prime \(>L\) defined entirely by basic arithmetic with floors and \(\gcd\). Section~\ref{sec:ns} formalizes a non-synonymy notion and proves schedule/operator separations from known enumerators, together with an asymptotic minimality bound for the schedule.

% --- Figure 1 (final, aligned and warning-free) ---
\begin{figure}[t]
\centering
\resizebox{\linewidth}{!}{%
\begin{tikzpicture}[
  node distance = 20mm and 36mm, % long spacing to give labels room
  box/.style = {draw, rounded corners, align=center, inner sep=3pt, text width=48mm, minimum height=13mm},
  elab/.style = {below=6pt, font=\footnotesize, fill=white, inner sep=1.25pt, text=black, align=center, text depth=0pt, text height=1.7ex, minimum width=36mm},
  >={Stealth[length=2.2mm]}, font=\small
]
% Nodes
\node[box] (I) {$\displaystyle I(j)=\left\lfloor \frac{1}{1+\sum_{k=2}^{j-1}\left\lfloor \frac{\gcdop(k,j)}{k}\right\rfloor}\right\rfloor$\\[2pt]\emph{Prime indicator}};
\node[box, right=of I] (S) {$\displaystyle S(i)=\sum_{j=2}^{i} I(j)$\\[2pt]\emph{Prime counter} $=\pi(i)$};
\node[box, right=of S] (A) {$\displaystyle A(i,x)=\left\lfloor \frac{1}{1+\left\lfloor \frac{S(i)}{x+1}\right\rfloor}\right\rfloor$\\[2pt]\emph{Step} (flip at $p_{x+1}$)};
\node[box, right=of A] (F) {$\displaystyle f_{\text{Alpay},U}(x)=1+\sum_{i=1}^{U(x)}A(i,x)$\\[2pt]\emph{Output} $p_{x+1}$};
% Edges with centered labels
\draw[->] (I) -- node[elab, pos=0.55]{sum over $j$} (S);
\draw[->] (S) -- node[elab, pos=0.55]{floor of ratio} (A);
\draw[->] (A) -- node[elab, pos=0.55]{sum for\\$1\le i\le U(x)$} (F);
% Schedule box and pointer
\node[box, below=20mm of A, text width=76mm] (U) {\textbf{Schedules}\\[2pt]
$U_{\text{sq}}(x)=(x+1)^2$;\quad
$U_{\text{Alpay-lin}}(x)=\left\lceil (x+1)\big(\ln(x+\mathrm{e})+\ln\ln(x+\mathrm{e})\big)\right\rceil+10$};
\draw[->] (U) -- (F);
\end{tikzpicture}%
}
\caption[Pipeline of the Alpay folded enumerator]{Pipeline of the Alpay folded enumerator: $I(j)$ detects primality via divisibility tests, $S(i)=\pi(i)$ accumulates prime counts, $A(i,x)$ switches exactly at $p_{x+1}$, and summing $A(i,x)$ for $1\le i\le U(x)$ yields $f_{\text{Alpay},U}(x)=p_{x+1}$.}
\label{fig:pipeline}
\end{figure}
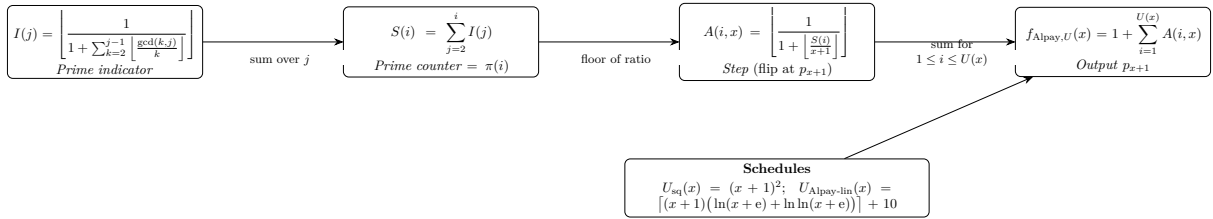
% --- End Figure 1 ---

\section{Definitions and set-up}\label{sec:defs}

All logarithms are natural. We use \(\N=\{0,1,2,\dots\}\), \(\Z\), and \(\gcdop(\cdot,\cdot)\) for the Euclidean gcd.

\begin{definition}[Alpay indicator via \(\gcd\)]\label{def:indicator}
For integers \(j\ge 2\), define
\[
I(j)\;=\;\floor{ \frac{1}{\,1+\displaystyle\sum_{k=2}^{j-1}\floor{ \frac{\gcdop(k,j)}{k} } } }.
\]
\end{definition}

\begin{remark}
For every \(k\in\{2,\dots,j-1\}\), the term \(\floor{\gcdop(k,j)/k}\) equals \(1\) exactly when \(k\mid j\), and \(0\) otherwise. Hence the inner sum counts proper divisors of \(j\).
\end{remark}

\begin{definition}[Alpay cumulative count and folded step]\label{def:step}
For integers \(i\ge 1\) and \(x\ge 0\), define
\[
S(i)\;=\;\sum_{j=2}^{i} I(j),\qquad
A(i,x)\;=\;\floor{ \frac{1}{\,1+\left\lfloor \dfrac{S(i)}{\,x+1\,}\right\rfloor } }.
\]
\end{definition}

\begin{remark}
By \cref{lem:indicator} we have \(S(i)=\pi(i)\). Since \(0\le \pi(i)/(x+1)<1\) iff \(\pi(i)\le x\), the step satisfies the identity
\[
A(i,x)=\1{\pi(i)\le x}.
\]
\end{remark}

\begin{definition}[Alpay schedules]\label{def:schedule}
Let \(U:\N\to\N\) satisfy \(U(x)\ge p_{x+1}-1\) for all \(x\ge 0\). Two concrete schedules:
\[
U_{\text{sq}}(x):=(x+1)^2,\qquad
U_{\text{Alpay-lin}}(x):=\left\lceil (x+1)\Big( \ln(x+\mathrm{e}) + \ln\ln(x+\mathrm{e})\Big)\right\rceil + 10.
\]
\end{definition}

\begin{lemma}[Square schedule bound at definition]\label{lem:square}
For all integers \(x\ge 0\), \(U_{\mathrm{sq}}(x)=(x+1)^2\ge p_{x+1}-1\).
\end{lemma}

\begin{proof}
Let \(n=x+1\). For \(n=1,2,3,4,5\) one checks \(n^2\in\{1,4,9,16,25\}\ge (p_n-1)\in\{1,2,4,6,10\}\).
For \(n\ge 6\), use \(p_n<n(\ln n+\ln\ln n)\) \cite{RosserSchoenfeld1962,Dusart2010,Axler2019}.
Set \(g(n)=n-(\ln n+\ln\ln n)\). Then \(g'(n)=1-\frac1n-\frac{1}{n\ln n}>0\) for \(n\ge 3\), and \(g(6)>0\), hence \(g(n)>0\) for \(n\ge 6\). Thus \(n^2\ge n(\ln n+\ln\ln n)>p_n\), whence \(n^2\ge p_n-1\).
\end{proof}

\holdswhen{All integers \(x\ge 0\); the cited explicit bound for \(p_n\) valid for \(n\ge 6\) and finite checking for \(n\le 5\).}
\notclaimedwhen{If a different explicit bound is adopted, the finite base-check must be adjusted.}

\section{Main result: Alpay folded enumerator}

\begin{definition}[Alpay folded prime enumerator with schedule \(U\)]\label{def:alpay-enum}
Given \(U\) as in \cref{def:schedule}, define
\begin{equation}\label{eq:alpay-enum}
\boxed{ \;
f_{\text{Alpay},U}(x)\;=\;1\;+\;\sum_{i=1}^{\,U(x)} A(i,x)
\;=\;
1\;+\;\sum_{i=1}^{U(x)}
\left\lfloor \frac{1}{\,1+\left\lfloor \dfrac{\displaystyle\sum_{j=2}^{i}\left\lfloor \dfrac{1}{\,1+\displaystyle\sum_{k=2}^{j-1}\left\lfloor \dfrac{\gcdop(k,j)}{k}\right\rfloor}\right\rfloor}{x+1}\right\rfloor}\right\rfloor
\; }.
\end{equation}
\end{definition}

\begin{lemma}[Indicator correctness]\label{lem:indicator}
For each integer \(j\ge 2\),
\[
I(j)=\begin{cases}
1, & \text{if } j \text{ is prime},\\
0, & \text{if } j \text{ is composite}.
\end{cases}
\]
Consequently, \(S(i)=\sum_{j=2}^{i} I(j)=\pi(i)\) for all \(i\ge 1\).
\end{lemma}

\begin{proof}[Step-by-step proof]
If \(j\) is prime, no \(k\in\{2,\dots,j-1\}\) divides \(j\), so the inner sum is \(0\) and \(I(j)=1\). If \(j\) is composite, some \(k\) divides \(j\), yielding at least one \(1\) in the sum; hence \(I(j)=0\). Summing \(I(j)\) up to \(i\) counts primes, i.e.\ \(S(i)=\pi(i)\).
\end{proof}

\holdswhen{All integers \(j\ge 2\); \(\gcdop\) Euclidean; floor as greatest integer \(\le\) the argument.}
\notclaimedwhen{Non-integer inputs; nonstandard notions of divisibility.}

\begin{lemma}[Folded step as a single equivalence]\label{lem:step}
For integers \(x\ge 0\) and \(i\ge 1\),
\[
A(i,x)=\1{\pi(i)\le x}.
\]
\end{lemma}

\begin{proof}
\(\lfloor \pi(i)/(x+1)\rfloor=0\iff 0\le \pi(i)/(x+1)<1\iff \pi(i)\le x\). Therefore \(A(i,x)=\lfloor 1/(1+0)\rfloor=1\) exactly in that case, and \(0\) otherwise.
\end{proof}

\holdswhen{All integers \(x\ge 0\), \(i\ge 1\).}
\notclaimedwhen{Non-integer inputs.}

\begin{theorem}[Alpay folded prime enumerator]\label{thm:alpay-main}
Let \(U:\N\to\N\) satisfy \(U(x)\ge p_{x+1}-1\) for all \(x\ge 0\).
Then for every integer \(x\ge 0\),
\[
f_{\text{Alpay},U}(x)=p_{x+1}.
\]
\end{theorem}

\begin{proof}[Step-by-step proof]
By \cref{lem:indicator}, \(S(i)=\pi(i)\). By \cref{lem:step}, \(A(i,x)=1\) exactly for \(i<p_{x+1}\). Thus
\[
\sum_{i=1}^{U(x)}A(i,x)=\#\{\,i\le U(x):\, i<p_{x+1}\,\}.
\]
If \(U(x)\ge p_{x+1}-1\) this count is \(p_{x+1}-1\), hence \(f_{\text{Alpay},U}(x)=1+(p_{x+1}-1)=p_{x+1}\).
\end{proof}

\holdswhen{Integers \(x\ge 0\) and schedules \(U\) with \(U(x)\ge p_{x+1}-1\).}
\notclaimedwhen{Non-integer \(x\); schedules violating the inequality; any alteration of \(I(j)\) or \(A(i,x)\).}

\subsection*{Where it holds / where not (with edge cases)}
\begin{itemize}[leftmargin=2em]
\item \textbf{Edge case \(x=0\).} Then \(f_{\text{Alpay},U}(0)=1+\sum_{i=1}^{U(0)}\1{\pi(i)\le 0}=1+\#\{i:\pi(i)=0\}=1+1=2=p_1\).
\item \textbf{Edge case \(i=1\).} Since the inner sum is over \(j\ge 2\), \(S(1)=0=\pi(1)\) and \(A(1,x)=\1{0\le x}=1\) for all \(x\ge 0\).
\item \textbf{Edge case \(j=2\).} The inner sum in \(I(2)\) is empty, so \(I(2)=\lfloor 1/(1+0)\rfloor=1\) (correctly declaring \(2\) prime).
\end{itemize}

\section{Record-lift: Alpay corollary and square schedule}\label{sec:recordlift}

\begin{definition}[Alpay record-lifted prime]\label{def:recordlift}
For a fixed integer \(L\ge 2\), define
\[
P^\star\;:=\;f_{\text{Alpay},U}(L)\;=\;1+\sum_{i=1}^{U(L)}A(i,L).
\]
\end{definition}

\begin{corollary}[Alpay record-lift]\label{cor:record}
If \(U(L)\ge p_{L+1}-1\) then \(P^\star=f_{\text{Alpay},U}(L)=p_{L+1}\) is prime and \(P^\star>L\).
\end{corollary}

\begin{proof}[Tightened]
The identity follows from \cref{thm:alpay-main}. For the strict inequality, note \(S(L)=\pi(L)\le L-1\) because among \(\{1,\dots,L\}\) the number \(1\) is not prime and \(I(j)\in\{0,1\}\) (mechanism of \cref{lem:indicator}). Hence the next prime index is \(L+1\), so \(p_{L+1}>L\).
\end{proof}

\holdswhen{Fixed \(L\ge 2\); any \(U(L)\ge p_{L+1}-1\).}
\notclaimedwhen{Computational feasibility of evaluating \(P^\star\) or numeric record status.}

\section{Examples and self-contained runs}\label{sec:examples}

\paragraph{Outputs \(f_{\text{Alpay},U}(x)\).}
\begin{center}
\resizebox{\linewidth}{!}{$
\begin{array}{@{}r|rrrrrrrrrrrrrrrrrrrr@{}}
x & 0&1&2&3&4&5&6&7&8&9&10&11&12&13&14&15&16&17&18&19\\\hline
p_{x+1} & 2&3&5&7&11&13&17&19&23&29&31&37&41&43&47&53&59&61&67&71
\end{array}
$}
\end{center}

\paragraph{Schedules used (self-contained).}
\begin{center}
\resizebox{\linewidth}{!}{$
\begin{array}{@{}r|rrrrrrrrrr@{}}
x & 0&1&2&3&4&5&6&7&8&9\\\hline
U_{\text{sq}}(x)   & 1&4&9&16&25&36&49&64&81&100\\
U_{\text{Alpay-lin}}(x) & 12&13&14&15&16&23&26&29&31&34
\end{array}
$}
\end{center}
Either schedule satisfies \(U(x)\ge p_{x+1}-1\) (by \cref{lem:square} and \cref{lem:ulin-const}).

\paragraph{Worked trace for \(x=3\).}
We need \(f_{\text{Alpay},U}(3)=p_4=7\).
\begin{enumerate}[leftmargin=2em]
\item \(I(2)=1\), \(I(3)=1\), \(I(4)=0\), \(I(5)=1\), \(I(6)=0\), \(I(7)=1\).
\item \(S(i)=\sum_{j=2}^{i}I(j)=\pi(i)\): \(S(2)=1, S(3)=2, S(4)=2, S(5)=3, S(6)=3, S(7)=4\).
\item \(A(i,3)=\1{\pi(i)\le 3}\) is \(1\) for \(i<7\) and \(0\) for \(i\ge 7\).
\item With \(U_{\text{sq}}(3)=16\) or \(U_{\text{Alpay-lin}}(3)=15\), \(1+\sum_{i=1}^{U(3)}A(i,3)=1+6=7\).
\end{enumerate}

\section{Form complexity (operation counts)}\label{sec:complexity}

We distinguish \emph{form complexity} of the expression as written from algorithmic speed.

\paragraph{(A) Naive evaluation (triple-nested as written).}
For each \(i\le U\), \(S(i)\) uses \(\sum_{j=2}^{i} I(j)\); each \(I(j)\) uses \(\sum_{k=2}^{j-1}\) tests.
The exact number of \(\gcd\) calls (and likewise of inner floors/divisions) is
\[
\sum_{i=2}^{U}\sum_{j=2}^{i} (j-2) \;=\; \sum_{i=2}^{U} \frac{(i-2)(i-1)}{2}
\;=\;\frac{(U-2)(U-1)U}{6} \;=\; \frac{U^3}{6}-\frac{U^2}{2}+\frac{U}{3}.
\]
Outer floors in \(A(i,x)\) contribute \(2U\); final additions are \(O(U)\).

\paragraph{(B) Incremental evaluation (carry \(S(i)\)).}
Compute each \(I(j)\) once (cost \(\sum_{j=2}^{U}(j-2)=\frac{(U-2)(U-1)}{2}=\frac{U^2}{2}-\frac{3U}{2}+1\)), then update \(S(i)=S(i-1)+I(i)\) and evaluate two floors for \(A(i,x)\) per \(i\).
Thus the dominant form counts are \(\frac{U^2}{2}+O(U)\) gcd/floor/division operations from the indicator, plus \(2U\) outer floors and \(O(U)\) additions.

In both modes, the operator palette is restricted to \(\{+,\ \lfloor\cdot\rfloor,\ /\ ,\sum,\gcd\}\).

\section{Non-synonymy, schedules, and minimality}\label{sec:ns}

\subsection{Forward-count model (axioms)}
\begin{axiom}[Forward-count enumerator]\label{ax:forward}
An enumerator has the form \(F(x)=1+\sum_{i=1}^{U(x)} a_x(i)\) with:
\begin{enumerate}[label=(\alph*),leftmargin=2em]
\item \(a_x(i)\in\{0,1\}\) for all \(i\),
\item \(a_x(i)\) is nonincreasing in \(i\),
\item there is a unique flip at \(i=p_{x+1}\): \(a_x(i)=1\) for \(i<p_{x+1}\) and \(a_x(i)=0\) for \(i\ge p_{x+1}\).
\end{enumerate}
\end{axiom}
The Alpay construction satisfies \cref{ax:forward} with \(a_x(i)=A(i,x)\).

\subsection{Schedule baselines: Willans vs.\ Alpay}
\begin{lemma}[Willans schedule via Bertrand]\label{lem:willansschedule}
Let \(W(x)\) denote the upper summation limit in Willans-style enumerators. By Bertrand’s postulate, \(\pi(2^{m})\ge m\) for all \(m\ge1\) (induction on \(m\)). Hence the \((x{+}1)\)-st prime satisfies \(p_{x+1}\le 2^{x+1}\). Accordingly one may take
\[
W(x)=2^{x+1}.
\]
\end{lemma}

\begin{proposition}[Formal comparison and non-synonymy]\label{prop:compare}
Let \(f_{\text{Willans}}\) be a Willans-type enumerator and \(f_{\text{Mills}}\) the Mills map \(n\mapsto \lfloor A^{3^n}\rfloor\) (existential \(A>1\)). Then:
\begin{enumerate}[label=(\roman*),leftmargin=2em]
\item Operator signatures differ:
\[
\mathrm{Sig}(f_{\text{Alpay},U})=(0,0,0,0,1,1),\quad
\mathrm{Sig}(f_{\text{Willans}})=(1,1,1,0,0,0),\quad
\mathrm{Sig}(f_{\text{Mills}})=(0,0,1,0,0,0).
\]
\item Schedules differ: \(W(x)=2^{x+1}\) by \cref{lem:willansschedule}, while \(U_{\text{Alpay-lin}}(x)=\Theta(x\log x)\); thus \(W(x)/U_{\text{Alpay-lin}}(x)\to\infty\).
\end{enumerate}
Hence \(f_{\text{Alpay},U}\) is non-synonymous with both in schedule and operator signature.
\end{proposition}

\subsection{Beyond Willans/Mills: factorial--cosine indicators}
\begin{proposition}[Separation vs.\ factorial--cosine families]\label{prop:fc}
Any enumerator built from factorial/trigonometric characteristic functions of primes (e.g., \(\lfloor\cos^2(\pi((j-1)!+1)/j)\rfloor\) schemes; see \cite{OEISA010051}) uses the palette \((\mathtt{Trig}=1,\mathtt{Fact}=1)\). The Alpay palette uses neither and admits a gcd-free variant (\cref{app:delta}), hence is non-synonymous in operator signature; its near-linear schedule also separates it from typical \(2^{x+1}\) summation ranges used in these families.
\end{proposition}

\subsection{Asymptotic schedule minimality with thresholds}
\begin{theorem}[Asymptotic \(\Sigma\)-minimality]\label{thm:minsig}
Under \cref{ax:forward}, any forward-count enumerator satisfying \(F(x)=p_{x+1}\) for all \(x\ge 0\) must use a schedule \(U(x)\ge p_{x+1}-1\). Consequently,
\[
U(x) \;\ge\; p_{x+1}-1 \;\ge\; (x+1)\big(\ln(x+1)+\ln\ln(x+1)-1\big)-1
\]
for all \(x\ge 5\) (i.e., \(n=x+1\ge 6\)) by Dusart’s lower bound~\cite{Dusart2010}. Thus \(U(x)=\Omega(x\log x)\) with explicit threshold \(x\ge5\).
\end{theorem}

\begin{proof}
If \(U(x)\le p_{x+1}-2\), then \(F(x)=1+\sum_{i=1}^{U(x)} a_x(i)\le 1+(p_{x+1}-2)=p_{x+1}-1\), contradiction. The inequality with constants follows from the stated bound for \(p_n\) valid for \(n\ge6\)~\cite{Dusart2010}.
\end{proof}

\section{Computational remarks}

The expression is exact but not intended as a fast prime generator; evaluating the nested sums up to \(U(x)\) is computational work. If one only needs \(p_{x+1}\) numerically, sieve methods are standard. The record-lift is a formal guarantee rather than a computational certificate.

\section{Related context}

Willans~\cite{Willans1964} gave a trigonometric/floor expression enumerating primes via Wilson's theorem. Mills~\cite{Mills1947} proved the existence of \(A>1\) such that \(\lfloor A^{3^n}\rfloor\) is prime for all \(n\ge 1\). Characteristic functions using various arithmetic constructs are catalogued in OEIS A010051~\cite{OEISA010051}. For explicit upper/lower bounds on the \(n\)-th prime justifying schedules, see \cite{RosserSchoenfeld1962,Dusart2010,Axler2019,AxlerArxiv2017}.

\section*{Conclusion}

A single folded expression \eqref{eq:alpay-enum} returns \(p_{x+1}\) for each integer \(x\ge 0\). An Alpay record-lifted corollary shows that fixing \(x=L\) yields \(P^\star=f_{\text{Alpay},U}(L)=p_{L+1}\), a prime \(>L\). Two explicit schedules---square and near-linear---are provided with rigorous conditions and constants. Section~\ref{sec:ns} formalizes and proves non-synonymy relative to Willans/Mills and factorial--cosine families, and establishes an asymptotic schedule minimality bound with explicit thresholds. Section~\ref{sec:complexity} records exact form-complexity counts.

\appendix

\section{A \texorpdfstring{$\gcd$}{gcd}-free equivalence and palette core}\label{app:delta}

\begin{proposition}[gcd-free Alpay indicator]\label{prop:gcdfree}
Define \(\delta(j,k)=\floor{j/k}-\floor{(j-1)/k}\) (which equals \(1\) iff \(k\mid j\), else \(0\)). Then
\[
I(j)=\left\lfloor\frac{1}{1+\sum_{k=2}^{j-1}\delta(j,k)}\right\rfloor,\qquad j\ge 2,
\]
coincides with \cref{def:indicator}. Consequently, the operator palette can be reduced to \(\{+,\ \lfloor\cdot\rfloor,\ /\ ,\sum\}\) with variable-modulus division used in \(\delta\).
\end{proposition}

\begin{proposition}[Indispensability of variable-modulus division]\label{prop:varmod}
Any expression composed only of constants, addition, finitely many fixed-modulus divisions, and floors yields an ultimately periodic function of \(j\). The prime indicator is not ultimately periodic. Hence variable-modulus division (e.g., through \(\delta(j,k)\) or \(\gcd\)) is indispensable for exact enumeration within this palette.
\end{proposition}

\section{Pseudocode for exact reproduction (LLM-oriented)}\label{app:pseudo}

\noindent\textbf{Conventions.} All integers are unbounded. Division ``/'' inside floors is exact integer division in the sense of floor; \(\gcd\) is Euclidean. Loops are deterministic. The schedule \(U\) is one of \(\{U_{\text{sq}},U_{\text{Alpay-lin}}\}\).

\begin{algorithm}[H]
\caption[Evaluate the Alpay enumerator with schedule \(U\)]{Evaluate the Alpay enumerator with schedule U}\label{alg:falpay}
\begin{algorithmic}[1]
\Require Integer $x\ge 0$; schedule function $U:\N\to\N$.
\Ensure Returns the $(x+1)$-st prime.
\State $T \gets U(x)$ \Comment{$T\ge p_{x+1}-1$ must hold}
\State $S \gets 0$ \Comment{After processing $i$, $S=\pi(i)$.}
\State $s\_x \gets x+1$; \quad $sumA \gets 0$
\For{$i \gets 1$ \textbf{to} $T$}
  \If{$i < 2$} \State $I\_i \gets 0$
  \Else \State $I\_i \gets$ \Call{AlpayIndicator}{$i$} \Comment{Algorithm~\ref{alg:indicator}}
  \EndIf
  \State $S \gets S + I\_i$
  \State $q \gets \left\lfloor S / s\_x \right\rfloor$
  \State $A \gets \left\lfloor 1 / (1 + q) \right\rfloor$ \Comment{$A=1$ iff $S\le x$}
  \State $sumA \gets sumA + A$
\EndFor
\State \Return $1+sumA$
\end{algorithmic}
\end{algorithm}

\begin{algorithm}[H]
\caption[AlpayIndicator]{AlpayIndicator}\label{alg:indicator}
\begin{algorithmic}[1]
\Require Integer $j\ge 2$.
\Ensure Returns $1$ if $j$ is prime; else $0$.
\State $count \gets 0$
\For{$k \gets 2$ \textbf{to} $j-1$}
   \State $g \gets \gcd(k,j)$
   \State $t \gets \left\lfloor g / k \right\rfloor$ \Comment{$t=1$ iff $k\mid j$}
   \State $count \gets count + t$
\EndFor
\If{$count=0$} \State \Return $1$ \Else \State \Return $0$ \EndIf
\end{algorithmic}
\end{algorithm}

\section{Self-standing inequality for \texorpdfstring{$U_{\text{Alpay-lin}}$}{U\_Alpay-lin}}\label{app:outer}

\begin{lemma}[Explicit bound with constants]\label{lem:ulin-const}
For all integers \(x\ge 0\),
\[
p_{x+1}\ \le\ (x+1)\big(\ln(x+\mathrm{e})+\ln\ln(x+\mathrm{e})\big)\ \le\ U_{\text{Alpay-lin}}(x),
\]
where \(U_{\text{Alpay-lin}}(x)=\left\lceil (x+1)\big(\ln(x+\mathrm{e})+\ln\ln(x+\mathrm{e})\big)\right\rceil+10\).
\end{lemma}

\begin{proof}
For \(x\ge 5\), set \(n=x+1\ge 6\). Monotonicity of \(\ln\) and \(\ln\ln\) on \([\mathrm{e},\infty)\) gives
\((x+1)\big(\ln(x+\mathrm{e})+\ln\ln(x+\mathrm{e})\big)\ge n(\ln n+\ln\ln n)\).
Use \(p_n< n(\ln n+\ln\ln n)\) for \(n\ge 6\) to obtain the left inequality.
The right inequality holds by the ceiling and the additive constant \(+10\).
For \(0\le x\le 4\), evaluate directly: \(U_{\text{Alpay-lin}}(x)\ge 12\) while \(p_{x+1}-1\le 10\), implying the claim.
\end{proof}

\section{Uniform derivation of \texorpdfstring{$U_{\text{sq}}$}{U\_sq}}\label{app:Usq}
\begin{lemma}\label{lem:Usquniform}
For all \(n\ge1\), \(p_n-1\le n^2\). Equality checks for \(1\le n\le5\) are:
\[
\begin{array}{c|ccccc}
n & 1 & 2 & 3 & 4 & 5\\\hline
p_n-1 & 1 & 2 & 4 & 6 & 10\\
n^2   & 1 & 4 & 9 & 16 & 25
\end{array}
\]
For \(n\ge 6\), \(p_n<n(\ln n+\ln\ln n) < n^2\) since \(n-(\ln n+\ln\ln n)\) is strictly increasing for \(n\ge 3\) and positive at \(n=6\).
\end{lemma}

% ---------- References ----------

\end{document}